\documentclass{article} 

\usepackage{fullpage}

\usepackage{amssymb, amsmath, amsthm} 

\newtheorem{thm}{Theorem}[section]    
\newtheorem{lem}[thm]{Lemma}          
\theoremstyle{definition}
\newtheorem{prop}[thm]{Proposition}

\usepackage{graphicx}
\usepackage{wrapfig, sidecap, mathrsfs, float, caption}

\newcommand{\PR}{\mathbb{P}}
\newcommand{\R}{\mathbb{R}}

\newcommand{\C}{\mathbb{C}}

\title{A Classification of subgroups of $SL(4,{\mathbb R})$ Isomorphic to ${\mathbb R}^3$ and Generalized Cusps in Projective 3 Manifolds 
}

\author{Arielle Leitner 
}

\begin{document}

\maketitle

\begin{abstract}This paper uses work of Haettel to classify all subgroups of $PGL_4(\R)$ isomorphic to $(\R^3 , +)$, up to conjugacy. We use this to show there are 4 families of generalized cusps up to projective equivalence in dimension 3.
\end{abstract}



There are 15 conjugacy classes of subgroups in $PGL_4(\R)$ isomorphic to $( \R ^3, +)$, see Theorem \ref{all15} for the precise list.  Classification of abelian subalgebras of $\mathfrak{sl}(n,\R)$ has long been of interest and closely related problems have been studied in \cite{Haettel}, \cite{IM}, \cite{SupTysh}, and \cite{WZ}.    The remainder of the paper applies this result to classify \emph{generalized cusps} in convex projective manifolds of dimension 3. 

Suppose $M$ is a manifold of dimension greater than 2.  By Mostow-Prasad rigidity, a finite volume hyperbolic structure on $M$ is unique up to isometry.  
The notion of a hyperbolic structure on a manifold may be generalized to a \emph{properly convex projective structure} as follows.   A 
convex set with non-empty interior $\Omega \subset \R P^n$ is \emph{properly convex} if the closure is disjoint from some projective hyperplane. 
A {\em properly convex manifold} is $M = \Omega / \Gamma$ where $\Gamma\subset PGL_{n+1} (\R)$  is discrete  and
acts freely on $\Omega$.  
 If $n=3$ such an $M$ is a \emph{generalized cusp} if 
 $M$ is diffeomorphic to $T^2 \times [0, \infty)$ and $\partial M$
 is   strictly convex (contains no line segment). 

In the case of convex projective structures, there is no notion of Mostow rigidity, so there is a richer deformation theory.  When $M$ is closed, Koszul \cite{Koszul}, shows that small deformations in the holonomy of hyperbolic structures yield properly convex projective structures. When $M$ is not compact, Koszul's result no longer holds. Cooper, Long and Tillmann  \cite{CoopLongTil}, have shown Koszul's result can be extended if certain conditions on the holonomy of each end are satisfied.   In dimension 3, this is as follows:

 The holonomy, $\Gamma$,
 of a generalized cusp is conjugate to a lattice in an upper triangular subgroup $H \subset PGL_4 (\R)$, and $H$ is called a \emph{cusp Lie group} in this paper.  
 These are characterized by the property that $H\cong \R^2$ 
and there is a point $x\in\R  P^3$ such that $H\cdot x\subset \R P^3$ is a strictly
  convex surface. Moreover $H$ is a subgroup of an upper triangular group $G \subset PGL_4(\R)$ with $G\cong \R^3$. 
  
 \begin{thm}\label{4cusps}Each cusp Lie group in $PGL_4 (\R)$ is conjugate to exactly one of the following groups:  
 $$\begin{array} {l}
C([r:s:t]) = \left\{ \left( \begin{array}{cccc}
 e^{a} & 0 & 0&0 \\
0 & e^{b} & 0 &0  \\
 0 & 0 & e^{c} &0\\
 0&0&0& e^{-a-b-c}  \end{array} \right)\ :  \begin{array}{c} a,b,c \in \R \\  ar +bs +ct =0 \end{array} \right \},
  \begin{array}{l} 
   [r:s:t] \in \R P^2  \\
r\geq s \geq t >0 
 \end{array}$$
 \vspace{.1in}\\

$$E(s)=\left\{\left( \begin{array}{cccc}
e^{b-a} &0&0&0 \\
 0 & e^a & e^a (bs + a) &0  \\
 0 & 0 & e^a &0\\
 0&0&0& e^{-a-b}  \end{array} \right) :\ a,b\in{\mathbb R} \right \}, \textrm{where }    0 \leq s < 1/2  $$

  \vspace{.1in}\\
 
$$F= \left\{\left( \begin{array}{cccc}
 e^a & e^a b & \frac{1}{2}e^a(b^2 + 2a)&0 \\
 0 & e^a & e^a b &0  \\
 0 & 0 & e^a &0\\
 0&0&0& e^{-3a}  \end{array} \right)  :\ a,b\in{\mathbb R} \right \}, $$


 \vspace{.1in}\\

$$N= \left\{\left( \begin{array}{cccc}
 1 & a& b & \frac{1}{2} (a^2 + b^2) \\
 0 & 1&0 &a  \\
 0 & 0 & 1 &b\\
 0&0&0& 1  \end{array} \right)  :\ a,b\in{\mathbb R} \right \}. 
\end{array} 
$$ 
\end{thm}

%
%

\begin{proof} [Sketch proof of Theorem \ref{4cusps}] 
The subgroups  $G\subset PGL_4(\R)$ isomorphic to $( \R ^3, +)$ are classified in Theorem \ref{all15}. 
For each such $G$, Proposition \ref{4groups} lists the 2-dimensional subgroups $H$ with a strictly convex orbit.  Proposition \ref{conjsubset}
determines which of these 2 dimensional groups are conjugate. 
\end{proof} 

It follows from \cite{CoopLongTil}, and it is easy to check directly, that every lattice in one of the Lie groups in Theorem \ref{4cusps} is the holonomy of a generalized cusp.
 It is shown in \cite{Ballas} and \cite{Gye-Seon} that there are properly convex projective structures on the Figure 8 knot complement
 with generalized cusps of types $C,N$ and $F$. At the time of writing it is not known if it also admits one of type $E$.

\section{Subgroups of $PGL_4(\R)$ Isomorphic to $(\R^3,+)$  }

The classification of 3-dimensional abelian subalgebras in $\mathfrak{gl}_4(\C)$ is given in \cite{SupTysh}, p.134, and in \cite{IM}, section 3.1.  The classification of maximal abelian subalgebras of $\mathfrak{sl}_4 (\R)$ is given as the main result of \cite{WZ}, but there are some of dimension larger than 3, and some with compact factors.   However, the author was unable to find a classification of 3-dimensional abelian subalgebras over $\R$. 

Let $G \leq PGL_{n+1}(\R)$ be a group, and $p \in \R P^n$.  
The \emph{orbit} of $p$ under $G$ is the set of images $\{ g.p : g \in G\}$.  
The orbits of $G$ acting on $\R P^n$ give a partition of $\R P^n$.  
 An \emph{orbit closure of $G$} is the {\em Zariski}-closure of an orbit of $G$. 
An orbit closure may contain an orbit of lower dimension.

\begin{thm}\label{all15}  In $PGL_4(\R)$ there are precisely 15 conjugacy classes of subgroups isomorphic to $(\R ^3, +)$: 
$$
\begin{array}{cccc} 
C& E_1 & F_0 & F_1\\
 \left( \begin{array}{cccc}
a & 0 & 0&0 \\
0 & b & 0 &0  \\
0 & 0 & c &0\\
0&0&0& \frac{1}{abc}  \end{array} \right) &
 \left( \begin{array}{cccc}
a & 0 & 0&0 \\
0 & b & c &0  \\
0 & 0 & b &0\\
0&0&0& \frac{1}{ab^2}  \end{array} \right)  &
 \left( \begin{array}{cccc}
a & b & 0&0 \\
0 & a & 0 &0  \\
0 & 0 & \frac{1}{a} &c\\
0&0&0& \frac{1}{a}  \end{array} \right) &
\left( \begin{array}{cccc}
a & b & c&0 \\
0 & a & b &0  \\
0 & 0 & a &0\\
0&0&0& \frac{1}{a^3}  \end{array} \right)  
\end{array}
$$
$$ \begin{array}{cccc}
F_2 & F_3 & N_1 & N_2 \\
\left( \begin{array}{cccc}
a & b & c&0 \\
0 & a & 0 &0  \\
0 & 0 & a &0\\
0&0&0& \frac{1}{a^3}  \end{array} \right) &
 \left( \begin{array}{cccc}
a & 0 & c&0 \\
0 & a & b &0  \\
0 & 0 & a &0\\
0&0&0& \frac{1}{a^3}  \end{array} \right) &
 \left( \begin{array}{cccc}
1 & a & b&c \\
0 & 1 & a &b  \\
0 & 0 & 1 &a\\
0&0&0& 1  \end{array} \right) &
  \left( \begin{array}{cccc}
1 & a & b&c \\
0 & 1 & a &0  \\
0 & 0 & 1 &0\\
0&0&0& 1  \end{array} \right) 
\end{array}$$
$$\begin{array}{cccc}
 N_3 &N_4&N'_4 & N_5\\
 \left( \begin{array}{cccc}
1 & 0 & 0&c \\
0 & 1 & a &b  \\
0 & 0 & 1 &a\\
0&0&0& 1  \end{array} \right) &
  \left( \begin{array}{cccc}
1 & a & b&c \\
0 & 1 & 0 &b  \\
0 & 0 & 1 &a\\
0&0&0& 1  \end{array} \right) &
  \left( \begin{array}{cccc}
1 & a & b&c \\
0 & 1 & 0 &a  \\
0 & 0 & 1 &b\\
0&0&0& 1  \end{array} \right) &
 \left( \begin{array}{cccc}
1 & 0 & b&c \\
0 & 1 & a &b  \\
0 & 0 & 1 &0\\
0&0&0& 1  \end{array} \right) 
\end{array}$$
$$\begin{array} {ccc} 
N_6 & N_7 & N_8\\
\left( \begin{array}{cccc}
1 & a & 0&c \\
0 & 1 & 0 &0  \\
0 & 0 & 1 &b\\
0&0&0& 1  \end{array} \right) &
 \left( \begin{array}{cccc}
1 & 0 & 0&c \\
0 & 1 & 0 &b  \\
0 & 0 & 1 &a\\
0&0&0& 1  \end{array} \right) &
 \left( \begin{array}{cccc}
1 & a & b&c \\
0 & 1 & 0 &0  \\
0 & 0 & 1 &0\\
0&0&0& 1  \end{array} \right) 
\end{array} $$
where each matrix represents a group by taking the union over all possible $a,b,c \in \R$ or $\R _+$, as appropriate. 
\end{thm} 

\begin{proof} Every subgroup $\exp(\mathfrak g) \leq SL_4(\R)$ isomorphic to $\R ^3$ is conjugate to an upper triangular group. One way
to see this is as follows. Winternitz and Zassenhaus (\cite{WZ} p.117)  classify maximal abelian subalgebras of $\mathfrak{sl}_4( \R)$.  
Under the exponential map one gives
$\R^2\times S^1$ which we can ignore.
The remainder give an upper triangular group
isomorphic to $\R^3$ or $\R^4$. Since $\mathfrak g$ is contained in some such maximal subalgebra the claim follows.

Haettel, \cite{Haettel} Proposition 6.1,  proves every  
3-dimensional abelian Lie subalgebra of the Borel subalgebra in $\mathfrak{sl}_4( \R)$ is  (up to conjugacy
in the Borel group)  one of 10 types. There are now two steps to the proof of the theorem.
Step 1: exponentiate each algebra in Haettel's list into $SL(4,\R)$, and show these groups are in the list above.
Step 2: show none of the groups in this list are conjugate. 

\noindent \textbf{Step 1:}  Exponentiate the algebras in Haettel's list. 

\begin{description} 
\item{Type 1:}  The Cartan subalgebra 
$$\mathfrak{a}= \left( \begin{array}{cccc}
a & 0 & 0&0 \\
0 & b & 0 &0  \\
0 & 0 & c &0\\
0&0&0&-a-b-c  \end{array} \right)$$ 
has $\textrm{exp}(\mathfrak{a})=C$. 

\item{Type 2:} These are algebras with three distinct weights and one off diagonal entry, which consist of matrices of the forms: 
$$ \mathfrak{i}_{\alpha}= \left( \begin{array}{cccc}
a & b & 0&0 \\
0 & a & 0 &0  \\
0 & 0 & c &0\\
0&0&0&-2a-c  \end{array} \right), 
\mathfrak{i}_{\beta}= \left( \begin{array}{cccc}
a & 0 & 0&0 \\
0 & b & c &0  \\
0 & 0 & b &0\\
0&0&0&-2b-a \end{array} \right), $$
$$ \mathfrak{i}_{\gamma}=\left( \begin{array}{cccc}
a &  & 0&0 \\
0 &-2b -a & 0 &0  \\
0 & 0 & b &c\\
0&0&0&b  \end{array} \right), 
\mathfrak{i}_{\alpha + \beta} =  \left( \begin{array}{cccc}
a & 0 & c&0 \\
0 & b & 0 &0  \\
0 & 0 & a &0\\
0&0&0&-2a-b  \end{array} \right), $$
$$\mathfrak{i}_{\beta + \gamma}=  \left( \begin{array}{cccc}
a & 0 & 0&0 \\
0 & b & 0 &c  \\
0 & 0 & -2b-a &0\\
0&0&0& b  \end{array} \right), 
 \mathfrak{i}_ {\alpha + \beta + \gamma} = \left( \begin{array}{cccc}
a & 0& 0&c \\
0 & b & 0 &0  \\
0 & 0 & -2a -b &0\\
0&0&0&a  \end{array} \right). $$
The $\mathfrak{i}_{\delta}$ are all conjugate in $\mathfrak{sl}_4 (\R)$ by permutation matrices, and $\textrm{exp}(\mathfrak{i}_\beta)= E_1$.

\item{Types 3 and 5:} These are algebras with 2 distinct weights.  
Let $[x: y ] \in \R P^1$ be fixed. Types 3 and 5 are algebras consisting of matrices of the forms: 
$$ \mathfrak{i}^{\alpha, \beta}_{[x:y]}=
 \left( \begin{array}{cccc}
c & ax & b&0 \\
0 & c & ay &0  \\
0 & 0 & c &0\\
0&0&0&-3c \end{array} \right), 
\mathfrak{i}^{ \beta, \gamma}_{[x:y]}=
 \left( \begin{array}{cccc}
-3c & 0 & 0&0 \\
0 & c & ax &b  \\
0 & 0 & c &ay\\
0&0&0&c \end{array} \right).$$
The algebras $\mathfrak{i}^{\alpha, \beta}_{[x:y]}$ and $\mathfrak{i}^{ \beta, \gamma}_{[x:y]}$ are conjugate in $\mathfrak{sl}_4 (\R)$ by a permutation matrix.  Note $\textrm{exp} (\mathfrak{i}^{\alpha, \beta}_{[0:1]}) = F_3$,  $\textrm{exp} (\mathfrak{i}^{\alpha, \beta}_{[1:0]}) = F_2$, and $\textrm{exp} (\mathfrak{i}^{\alpha, \beta}_{[x:y]})$ is conjugate to $F_1$, for $(x,y) \neq (0,0)$. 

\item{Type 4:} This algebra has 2 distinct weights.  
It consists of matrices of the form 
$$\mathfrak{i}^{\alpha, \gamma}=
 \left( \begin{array}{cccc}
a & b & 0&0 \\
0 & a & 0 &0  \\
0 & 0 & -a &c\\
0&0&0&-a \end{array} \right), $$
and $ \textrm{exp}(\mathfrak{i}^{\alpha, \gamma})= F_0$. 

\item{Type 6:}  Let $[x: y : z ]  \in \R P^2$ be fixed, with $x, z \neq 0$, and consider the algebra consisting of matrices of the form
$$\mathfrak{i}_{[x:y:z]}= \left( \begin{array}{cccc}
0 & ax & bx&c\\
0 & 0 & ay &bz  \\
0 & 0 & 0 &az\\
0&0&0&0 \end{array} \right). $$
When $ y \neq 0 $, the group  $\textrm{exp}(\mathfrak{i}_{[x:y:z]})$ is conjugate to $N_1$, and $\textrm{exp}(\mathfrak{i}_{[x:0:z]})$ is conjugate to $N_4$ by a diagonal matrix. 

\item{Type 7:} Let $(y,t) \in \R ^2$ be fixed, and consider the algebra consisting of matrices of the form 
$$ \mathfrak{i}_{\alpha, y, t}= \left( \begin{array}{cccc}
0 & a & b&c\\
0 & 0 & ay &at  \\
0 & 0 & 0 &0\\
0&0&0&0 \end{array} \right). $$
If $(y,t) \neq (0,0)$, the group $\textrm{exp}(\mathfrak{i}_{\alpha, y, t})$ is conjugate to $N_2$ by an elementary matrix, and $\textrm{exp}(\mathfrak{i}_{\alpha, 0, 0})= N_8$. 

\item{Type  8:} Let $(y,t) \in \R ^2$ be fixed, and consider the algebra consisting of matrices of the form
$$ \mathfrak{i}_{\gamma, y, t}= \left( \begin{array}{cccc}
0 & 0 & at&c\\
0 & 0 & ay &b  \\
0 & 0 & 0 &a\\
0&0&0&0 \end{array} \right). $$
If $(y,t) \neq (0,0)$, the group $\textrm{exp}(\mathfrak{i}_{\gamma, y, t})$ is conjugate to $N_3$ by an elementary matrix, and $\textrm{exp}(\mathfrak{i}_{\gamma, 0, 0})= N_7$. 

\item{Type 9:} Let $[ x:y:z:t] \in \R P^3 $ be fixed, and consider the algebra consisting of matrices of the form
$$ \mathfrak{i}_{[x:y:z:t]}= \Bigg\{ \left( \begin{array}{cccc}
0 & 0 & b&c\\
0 & 0 & a &d  \\
0 & 0 & 0 &0\\
0&0&0&0 \end{array} \right) \Bigg| ax + by + cz + dt =0 \Bigg\}.$$
By conjugating by an elementary matrix, it is easy to check $\textrm{exp}(\mathfrak{i}_{[x:y:z:t]})$ is conjugate to $N_6$ if $[x:y:z:t] \in \{ [0:0:1:t], [1:y:0:0], [1:0:z:0],[0:1:0:t] \}$.  For example, $\textrm{exp}(\mathfrak{i}_{[0:0:1:t]})$ is conjugate to $N_6$ by $ -t \cdot E_{14}$. 
Otherwise, $\textrm{exp}(\mathfrak{i}_{[x:y:z:t]})$ is conjugate to $N_5 $. 

\item{Type 10:} Let $(x,y) \in \R ^2 $ be fixed, and consider the algebra consisting of matrices of the form 
$$\mathfrak{i}_{x,y}= \left( \begin{array}{cccc}
0 & a& by&c\\
0 & 0 & 0 &ax  \\
0 & 0 & 0 &b\\
0&0&0&0 \end{array} \right).$$
 If $sign(x)=sign(y)$ then $\textrm{exp}(\mathfrak{i}_{x,y})$ is conjugate to $N_4$, and if $sign(x)=-sign(y)$, then $\textrm{exp}(\mathfrak{i}_{x,y})$ is conjugate to $N'_4$.  
 Finally,  $\textrm{exp}(\mathfrak{i}_{0,0})= N_6$, $\textrm{exp}(\mathfrak{i}_{x,0})= N_2$ for $x \neq 0$, and $\textrm{exp}(\mathfrak{i}_{0,y})= N_3$ for $y \neq 0$. 
\end{description}
Thus every one of Haettel's algebras exponentiates to be in our list. \\
\noindent \textbf{Step 2:} Show none of the groups are conjugate  
by computing orbit closures.  Except in the case of $N_4$ and $N_4'$
 we prove  these 15 groups are in distinct conjugacy classes by showing they have orbit closures which are not projectively equivalent (\cite{LeitnerSLn}, Lemma 5).  Every orbit closure is a projective subspace of some dimension from $0$ to $3$. 

Let $\{e_1, ... e_4\}$ be the usual basis for $\R ^4$, and let $\{[e_1], ... [e_4]\}$ be the projective images in $\R P^3$. Here is a table of the orbit closures of each group.  If a subspace appears in a column for a certain dimension, this  means that the orbit closure of any point in that subspace has that dimension.  For example $\langle e_1, e_2 \rangle $ in the dimension 0 column means every point on the line is fixed.   If  ``$\langle e_1, x \rangle$ for $x \in \langle e_1, e_2, e_3 \rangle$,'' appears in the dimension 1 column, it means every line through $[e_1]$ contained in the plane $\langle e_1, e_2, e_3 \rangle$ is an orbit closure. 
$$\begin{array}{|c||c|c|c|c|}
\hline 
\textrm{Group} & \dim 0 & \dim 1 & \dim 2 & \dim 3 \\
\hline
\hline 
C 
& \begin{array}{c} [e_1], [e_2], \\
\lbrack e_3 \rbrack, [e_4]  \end{array}
& \begin{array}{c}  \langle e_1, e_2 \rangle ,  \langle e_1, e_3 \rangle,  \langle e_2, e_3 \rangle, \\
 \langle e_1, e_4 \rangle, \langle e_2, e_4 \rangle, \langle e_3, e_4 \rangle 
 \end{array} 
 &\begin{array}{c} \langle e_1, e_2 , e_3 \rangle,  \langle e_1, e_2 , e_4 \rangle,\\
  \langle e_1, e_3 , e_4 \rangle, \langle e_2, e_3 , e_4 \rangle 
  \end{array} & \R P^3 \\
\hline 
E_1 
& \begin{array} {c} [e_1], [e_2], \\
\lbrack e_4 \rbrack \end{array}
& \begin{array}{c}  \langle e_1, e_2 \rangle, \langle e_2, e_3 \rangle, \\
\langle e_1, e_4 \rangle, \langle e_2 , e_4 \rangle
\end{array} & 
\begin{array}{c}  \langle e_1, e_2, e_3 \rangle , \langle e_1, e_2, e_4 \rangle, \\
\langle e_2, e_3, e_4 \rangle 
\end{array} & \R P^3 \\
\hline 
F_0 & [e_1], [e_3] & \langle e_1 , e_3 \rangle, \langle e_1, e_2 \rangle , \langle e_3, e_4 \rangle & \langle e_1, e_2, e_3 \rangle , \langle e_1, e_3 , e_4 \rangle  & \R P^3 \\
\hline 
F_1 & [e_1], [e_4] & \langle e_1, e_4 \rangle, \langle e_1, e_2 \rangle & \langle e_1, e_2, e_3 \rangle, \langle e_1, e_2 , e_4 \rangle & \R P^3 \\
\hline 
F_2 & [e_1] , [e_4] &\begin{array}{c} \langle e_1, e _4 \rangle ,\\
 \langle e_1, x \rangle \textrm{ for } x \in \langle e_1, e_2, e_3 \rangle \end{array} 
 & \langle e_1, x , e_4 \rangle \textrm{ any } x \in \R P^3 &  \\
 \hline 
 F_3 & [e_4], \langle e_1, e_2 \rangle & \langle e_1, e_4 \rangle & \langle e_1, e_2 , e_3 \rangle, \langle e_1, e_2 , e_4 \rangle & \R P^3 \\
 \hline 
 N_1 & [e_1] & \langle e_1, e_2 \rangle & \langle e_1, e_2, e_3 \rangle & \R P^3 \\
 \hline 
 N_2 & [e_1] & \langle e_1, x \rangle \textrm{ for } x \in \langle e_1, e_2, e_3 \rangle & \langle e_1, e_2 , x \rangle \textrm{ any } x \in \R P^3 & \\
 \hline 
 N_3 & \langle e_1, e_2 \rangle & \langle e_2, x \rangle \textrm{ for } x \in \langle e_1, e_2, e_3 \rangle &  & \R P^3 \\
 \hline 
 N_4  & [e_1] & \langle e_1, x \rangle \textrm{ for } x \in \langle e_1, e_2 , e_3 \rangle & & \R P^3 \\
 \hline 
  N_4 ' & [e_1] & \langle e_1, x \rangle \textrm{ for } x \in \langle e_1, e_2 , e_3 \rangle & & \R P^3 \\
 \hline 
 N_5 & \langle e_1, e_2 \rangle & & \langle e_1, e_2, x \rangle \textrm{ any } x \in \R P^3 &  \\
 \hline 
 N_6 & \langle e_1 , e_3 \rangle & \langle e_1, x \rangle \textrm{ for } x \in \langle e_1, e_2 , e_3 \rangle & \langle e_1, e_3 , x \rangle \textrm{ any } x \in \R P^3 & \\
 \hline 
 N_7 & \langle e_1, e_2, e_3 \rangle & && \R P^3 \\
 \hline 
 N_8 & [e_1] & \langle e_1 , x \rangle \textrm{ any } x \in \R P^3  & & \\
 \hline 
\end{array}
$$

None of the orbit closures of the groups in the list are projectively equivalent, except $N_4$ and $N_4'$, which are not conjugate by Lemma \ref{N4}. Thus these are all conjugacy classes of subgroups of $SL_4 (\R)$ isomorphic to $\R ^3$. 
\end{proof} 

\begin{lem}\label{N4} The groups $N_4$ and $N'_4$ are not conjugate in $PGL_4 (\R)$, but they are conjugate in $PGL_4(\C)$.  
\end{lem} 
\begin{proof}  Consider the respective Lie algebras: 
$$ \mathfrak{N_4} = 
\left( \begin{array}{cccc}
0 & a& b&c\\
0 & 0 & 0 &b  \\
0 & 0 & 0 &a\\
0&0&0&0 \end{array} \right)
\textrm{  and  } 
\mathfrak{N'_4} = 
\left( \begin{array}{cccc}
0 & a& b&c\\
0 & 0 & 0 &a  \\
0 & 0 & 0 &b\\
0&0&0&0 \end{array} \right).
$$
To show $\mathfrak{N}_4$ and $\mathfrak{N}'_4$  are non-isomorphic Lie algebras, consider their images under the exponential map. Notice $\textrm{exp}(\mathfrak{N_4})$ has $ c+ \frac{ab}{2}$ in the upper right corner, and $\textrm{exp}(\mathfrak{N'_4})$ has $c + \frac{a^2+b^2}{2}$ in the upper right. The subalgebra of $\mathfrak{N_4}$ with $a=0$ is a 2-dimensional subalgebra which exponentiates linearly, i.e., every matrix entry is linear in the image under the exponential map.   There is no 2-dimensional Lie subalgebra of $\mathfrak{N'_4}$ which exponentiates linearly (since $a^2 + b^2$ is positive definite as a real quadratic form). 

  The groups $N_4$ and $N'_4$ are conjugate by a complex matrix, but \emph{not} a real matrix.   Over $\C$, the algebra $\mathfrak{N'_4}$ has a 2-dimensional subalgebra which exponentiates linearly, when $a = i b$. 
 \end{proof} 
 
  Iliev and Manivel prove there are 14 conjugacy classes of 3-dimensional abelian subalgebras in $\mathfrak{sl}_4(\C)$, see \cite{IM} section 3.1.  Their list is the same as in Theorem \ref{all15}, with only one representative for the conjugacy class $\{\mathfrak{N}_4, \mathfrak{N}'_4\}$ over $\C$.  

\section{Description of cusp Lie subgroups of $E_1$}

This section determines which 2-dimensional subgroups of the groups in Theorem \ref{all15} have a strictly convex orbit.  We do the case of $E_1$ in detail:  we first produce a 2 parameter family, $E(r,s)$, of cusp Lie groups, and then Proposition \ref{conjsubset} shows they are all conjugate to the groups $E(s)$ in Theorem \ref{4cusps}.

Ballas describes the cusps  arising from $N$ (the standard cusp), and $F$ in \cite{Ballas}.  Gye-Seon Lee has described the family of cusps  arising from $C([r:s:t])$.  
We follow the notation and ideas outlined in \cite{Ballas}.   
 
 Recall the \emph{second fundamental form} is a symmetric bilinear form on the tangent plane of a smooth surface in three-dimensional Euclidean space (see \cite{Morgan}).   It is given explicitly for the image of a twice continuously differentiable function  $f: \R^2 \to \R^3$ 
which is tangent to the $xy$ plane at the origin by
 $$\mathrm{I\!I}(f)= \frac{\partial ^2 f}{\partial x^2 } d x^2 + 2 \frac{\partial ^2 f}{\partial x \partial y } dx dy +  \frac{\partial ^2 f}{\partial y^2 } d y^2. $$
This gives the curvature of $f$ at the origin.  

The \emph{Gauss curvature}, $G$,  is the determinant of $\mathrm{I\!I}(f)$ (see \cite{Morgan} p.13).  Let $p$ be a point on the surface $S \subset \R ^3$, which is the image of a twice differentiable function $f : \R ^2 \to \R ^3$.  Proposition 3.5 in \cite{Morgan} says the second fundamental form at $p$, written $\mathrm{I\!I}(f)_p$, is similar to 
$$   g^{-1} \left[ \begin{array}{cc} 
\frac{\partial ^2 f } { \partial x^2 } \cdot \vec{n}   &  \frac{\partial ^2 f } { \partial x \partial y } \cdot \vec{n} \\
 \frac{\partial ^2 f } { \partial x \partial y } \cdot \vec{n}  & \frac{\partial ^2 f } { \partial y^2 } \cdot \vec{n}  \end{array} \right] , $$
 where $\vec{n}$ is the normal vector to $S$ at $p$, and $g$ is the metric. 
Proposition 3.5 in \cite{Morgan} also implies the sign of the curvature at $p$ depends only on the sign of $\det \mathrm{I\!I}(f)_p$. 
Therefore if $\det \mathrm{I\!I}(f)_p$ is positive, then $S$ is convex at $p$. 

Recall a surface is \emph{convex at a point} if it lies completely on one side of the tangent plane at that point. A surface is \emph{convex everywhere} if it there is a unique supporting hyperplane at every point, and at each point, the surface lies completely on one side of that hyperplane. 
Suppose there is a transitive affine group action on $S\subset {\mathbb R}^3$.  
Since affine maps preserve convexity, $S$ is convex everywhere if there is one point at which $S$ is convex.  

\begin{prop}\label{fund_form}  Suppose a surface, $S$, is the image of $f:{\R^2}\longrightarrow\R^3$
and $S$ is the orbit of a point under the action of an affine group, $G$.  
\begin{enumerate}
\item If there exists $p\in S$ with $\det \mathrm{I\!I}(f)_p>0$, then $S$ is convex everywhere. 
\item  If $S$ is convex everywhere, then $\mathrm{I\!I}(f)$ is positive definite. 
\end{enumerate}
\end{prop} 

\begin{proof}  The preceding discussion proves the first assertion.  To prove the second claim, set $k(p) := \mathrm{I\!I}(f)_p$, the numerical value of the second fundamental form at a point.  Let $A \in G$.  Then $A.S = S$.  Since $A$ is affine, $A$ multiplies $k(p)$ by a nonnegative scalar.  So $k (A.p ) =0$ if and only if $k(p)=0$.  Thus if $k(p)=0$ for some $p$, then $k \equiv 0$ everywhere.  But if $ k=0$, then $S$ has zero Gauss curvature, and $S$ contains a straight line.  But this contradicts that $S$ has a unique supporting hyperplane at every point.  Therefore $k(p)$ is positive everywhere. 
\end{proof}



Given $[r:s] \in \R P^1$, define 
$$E(r,s) := \left \{ \left( \begin{array}{cccc}
e^{a-b} & 0 & 0&0\\
0 & e^b & e^b(ar +bs) &0  \\
0 & 0 & e^b &0\\
0&0&0&e^{-b-a} \end{array} \right) : a,b \in \R  \right\} .$$
Then $E(r,s) \cong (\R ^2, +)$ is a subgroup of $E_1$.

\begin{prop}\label{E1}  
The cusp Lie groups contained in $E_1$ are the subgroups $E(r,s)$ with $|s| < |r|/2$. 
\end{prop}

\begin{proof} 
Every 2-dimensional Lie subalgebra of $\textrm{Lie} (E_1)$ is defined by an equation $ ar+ bs +ct=0$, where $r, s,t \in \R$ are fixed, and at least one of $r,s,t$ is not zero.  
Assume first $t \neq 0$, and scale so that $t=-1$, so $c= ar +bs$.   Exponentiating gives the 
 Lie group $E(r,s)$. 

Let $p=[x_0:y_0:z_0:1] \in \R P^3$.  The orbit of $p$ under $E(r,s)$ is the surface
$$S : = \{[e^{a-b} x_0: e^b y_0 + e^b(ar+bs)z_0: e^b z_0: e^{-a-b}] : a,b \in \R \}.$$
Scale by $e^{-a-b}$ so 
$$S=\{ [\frac{e^{a-b} x_0}{e^{-a-b}}:  \frac{e^b y_0 + e^b(ar+bs)z_0}{e^{-a-b}}: \frac{e^b z_0}{e^{-a-b}}:1] : a,b \in \R\},$$
and $S$ is in the affine patch that is the complement of the hyperplane $[ \ast:\ast:\ast:0]$. Moreover, $E(r,s)$ preserves the 3-dimensional affine subspace $\{ [x_1 : x_2 : x_3 :1] \}$, so view $E(r,s)$ as affine transformations on $\R ^3$.     Consider the map $f_p: \R^2 \to \R ^3$ given by  
$$f_p(a,b) =(e^{2a} ,  e^{a+2b}  + e^{a+2b}(ar+bs), e^{a+2b} ).$$   
Then $S$ is the image of $f_p$. 
Perform the coordinate change 
$A=e^{2a}, B=e^{a+2b}$. So
$$S=  \{(A, B(1+  r (\frac{1}{2}\ln A)  + s(\frac{-1}{4}\ln A + \frac{1}{2} \ln B)  ), B )| A, B \in \R_{>0}\}.$$ 
Then $S$ is the graph of $g_p(A,B) = B(1+  r (\frac{1}{2}\ln A)  + s(\frac{-1}{4}\ln A + \frac{1}{2} \ln B)  ) \subset \R ^3$. 

The determinant of
the second fundamental form is 
$\det \mathrm{I\!I}(g)_p =\frac{ (r^2 - 4s^2)}{16 A^2}$, 
which is positive when 
$| s| <| r|/2$. 
So by Proposition \ref{fund_form}, ${E}(r,s)$ has a convex orbit if and only if $|s| < |r|/2$.   Doing the same computation with $r$ (or $s \neq 0$) and $b= ar + ct$ yields analogous results.  Permuting coordinates gives equivalence of the fundamental forms. 
\end{proof} 

Recall the upper half space model of hyperbolic space gives a coordinate system with a  \emph{point at infinity}.   A generalization of these coordinates for properly convex domains is introduced in \cite{CoopLongTil1}.  Let $\Omega$ be a properly convex domain, $p$ a point in $\partial \Omega$, and $H$ a supporting hyperplane containing $p$.  There is an identification of the affine patch $\R P^n - H$  with $\R^n$  in which lines through $p$ not contained in $H$ are parallel to the $x_1$ axis.   This is achieved by applying a projective change of coordinates which sends $p$ to $[e_1]$ and $H$ to the projective hyperplane dual to $[e_{n+1}]$.   The $x_1$ direction is called the \emph{vertical direction}.   A set of coordinates with this property is called \emph{parabolic coordinates centered at $(H,p)$}, or just parabolic coordinates if $H$ and $p$ are clear from the context. 

\emph{Algebraic horospheres} are defined using parabolic coordinates as follows:  Let $t>0$, and define $\mathcal{S}_t$ as the translation of the part of $\partial \Omega$ that does not contain any line segments through $p$ by the vector $t e_1$.  These sets are \emph{algebraic horospheres centered at $(p,H)$}.  They coincide with Buseman horospheres at $\mathcal{C}^1$ points, see \cite{Bus}. Algebraic horospheres are homeomorphic to spheres with a point removed from the boundary. We will show cusp Lie groups act on convex sets to preserve a foliation by algebraic horospheres. The rays in $\Omega$ asymptotic to $p$ give a transverse foliation. See \cite{CoopLongTil1} for more on algebraic horospheres.

\begin{prop}  If $|s| < |r|/2$, the cusp Lie group ${E}(r,s)$ acts on a convex set foliated by algebraic horospheres, each of which is a convex surface preserved by the action of ${E}(r,s)$. 
\end{prop}

\begin{proof} 
Recall from the proof of Proposition \ref{E1} that the graph of $g_p$ gives an orbit of $E_1$ that is convex.  We rewrite this as  
$$\Omega = \{(x_1, x_2, x_3)\in \R ^3 | x_1, x_3 >0, (1+\frac{1}{4}(2r-s) \ln x_1 + \frac{1}{2}s \ln x_3) > x_2) \}$$
 is a convex set preserved by the action of $E(r,s)$.   
Let $\mathcal{H}_k$ be the orbit of $(0,k,0)$ under $E(r,s)$.   So, $\mathcal{H}_k$ is the graph of the strictly convex function 
$$h(r,s)=(1+\frac{1}{4}(2r-s) + \frac{1}{2}s ) \ln k =  \frac{1}{4}(4+ 2r+s)  \ln k.$$
Then $ \bigcup_{k>0} \mathcal{H}_k$ is a foliation of $\Omega$ by horospheres around the point $(0,1,0)$. 

Let $\Gamma$ be a lattice in $\Omega$.  Then $\Omega / \Gamma$ is a generalized cusp, diffeomorphic to $T^2 \times [0, \infty)$, by a diffeomorphism which sends $ \mathcal{H}_k/ \Gamma  \to T^2 \times \{k\}$.  The map $[x_1: x_2 : x_3 :1 ] \to (x_1, x_3)$ restricted to $\mathcal{H}_k$  is a developing map  for an affine structure on $\mathcal{H}_k / \Gamma$. 
\end{proof}

 \section{ Convex Orbits }
 
 This section determines which of 
the 15 groups in Theorem \ref{all15} have a subgroup that is a cusp Lie group. 
We use the methods of Proposition \ref{E1}. 


Suppose $\theta: \R ^3 \to \mathfrak{g}$ is an isomorphism of Lie algebras, so $\theta(a,b,c) = \mathfrak{g}$, a Lie algebra.  Given $[r:s:t] \in \R P^2$, define the subalgebra $\mathfrak{g}[r:s:t] := \theta \{ (a,b,c) \in \R ^3 : ra +sb +tc =0 \}$.  Every 2-dimensional subalgebra of $\mathfrak{g}$ is obtained this way.  Set $G[r:s:t]= \exp \mathfrak{g}[r:s:t]$. 

  \begin{prop}\label{4groups} 
  Suppose $G$ is one of the groups in Theorem \ref{all15}, and $H$ is a cusp Lie subgroup of $G$.  Then $G$ is one of $C, E_1, F_1$ or $N_4 '$, and $H$ is conjugate in $PGL_4(\R)$ to one of 
  \begin{itemize} 
 \item  $C[r:s:t]$ with $rst(r+s+t)>0$, 
 \item $E_1 [r:s:-1]= E(r,s)$ with $ |s| <| r|/2$, 
 \item  $F_1[r:s:-1]  $ with $r >0$ 
 \item  $N_4' [r:s:-1]$. 
  \end{itemize} 
    \end{prop} 

\begin{proof} 

  Since every 2-dimensional subalgebra of a 3-dimensional Lie algebra is obtained as $\mathfrak{g}[r:s:t]$, all possible subgroups of 
the groups in Theorem \ref{all15} isomorphic to $(\R^2, +)$ are realized as $G[r:s:t]$. 
Set $H := G[r:s:t]=\exp \mathfrak{g}[r:s:t] \cong (\R^2, +)$. 
 Let $p=[x_0: y_0: z_0:1] \in \R P^3$, and let $ S:= \{h.p: h \in H\}$ be the orbit of $p$. 
 The elements of $H$ preserve the affine patch that is the complement of $\{ [x_1 : x_2 : x_3 :0] \}$, so regard $H$ as a set of affine transformations of $\R ^3$, and $S$ as a surface that is the image of an orbit of $H$.  As in Proposition \ref{E1}, use projective equivalence to scale so $S$ consists of points  of the form  $[\ast : \ast: \ast:1]$, and dehomogenize (and perhaps perform a change of coordinates) so $S$ is the graph of a function $g: \R ^2 \to \R ^1$. 
 By Proposition \ref{fund_form}, we know $\det \mathrm{I\!I}( g)_p >0$, if and only if $g$ is convex. 
We wrote a Mathematica program to execute this process, and ran it on all 
subgroups isomorphic to $(\R^2, +)$, of 
the groups in Theorem \ref{all15}. 
 
 The 
 second fundamental form depends on $r,s$, and the point $p$. The only groups which give rise to positive definite second fundamental forms are $C, E_1, F_1$, and $N_4'$. 
 Below is a chart showing the results of the computations for each of the groups.   The computations  are similar to those for $E_1$ given in Proposition \ref{E1}.
 
$$\begin{array}{|c|c|}
\hline
\textrm{Group} &\begin{array}{c} \textrm{Orbit of Generic Point $p=[x_0: y_0: z_0 :1]$} \\ 
\det \mathrm{I\!I}(g)_p \textrm{ [up to scaling] }: r,s,t \textrm{ are fixed and } a,b,c \textrm{ are variables} \end{array}\\
\hline
C & \begin{array}{c} 
(e^{\frac{ar}{-ar-bs-ct}} x_0,e^{\frac{bs}{-ar-bs-ct}} y_0,e^{\frac{ct}{-ar-bs-ct}}z_0)\\
rst(r+s+t) x_0^2 y_0^2 z_0^2 
\end{array}\\
\hline
E_1 &\begin{array}{c} (e^a x_0, e^b y_0 + e^b(ar+bs)z_0, e^b z_0, e^{-a-2b})\\ 
16 ( r-2s)(r+2s) x_0^2 z_0^4 \end{array} \\
\hline
F_0 & \begin{array}{c}(e^a x_0 + be^a y_0, e^a y_0, e^{-a}z_0 + e^{-a}(ar + bs), e^{-a})\\ -16 r^2 y_0^4\end{array}\\
\hline
F_1 & \begin{array}{c}(e^a x_0 + e^a (b s + a r)z_0, e^a y_0 + be^a z_0, e^a z_0, e^{-3a}) \\ 64 r z_0^6 \end{array}\\
\hline
F_2 & \begin{array}{c} (e^a x_0 + e^a b y_0 + e^a (ar +bs) z_0, e^a y_0 , e^a z_0, e^{-3a})\\0 \end{array} \\
\hline
F_3 &  \begin{array}{c} (e^a x_0 + b e^a y_0 + \frac{1}{2} e^a (b^2 + 2 b s + 2 a r)z_0, e^a y_0 + b e^a z_0, e^a z_0, e^{-3a}) \\  0 \end{array}\\
\hline
N_1& \begin{array}{c} (x_0 +a y_0 + ( \frac{a^2}{2} +b) z_0 +( \frac{a^3}{6} + b s + a (b + s), y_0 + az_0 ( \frac{a^2}{2} +b)  , z_0+a,1) \\ -1 \end{array} \\
\hline
N_2 & \begin{array}{c} (x_0 + a y_0 + (\frac{a^2}{2}+b) z_0 + ar+bs, y_0 + a z_0, z_0, 1) \\ 0 \end{array} \\
\hline 
N_3 &\begin{array}{c} (x_0+ (ar+bs) z_0,  y_0 +az_0 +  (\frac{a^2}{2}+b) ,  z_0+a, 1) \\   0 \end{array} \\
\hline
N_4 & \begin{array}{c}(x_0 + a y_0 + bz_0 +(ar+bs) ab ,y_0 + b, z_0+a,1 ) \\-1 \end{array} \\ 
\hline
N_4 ' &  \begin{array}{c}(x_0 + a y_0 + bz_0+(ar+bs)  \frac{1}{2}(a^2 + b^2), y_0+a, z_0+b,1 ) \\1 \end{array} \\
\hline
N_5 & \begin{array}{c} (x_0 + a z_0 + ar +bs, y_0 + bz_0, z_0, 1) \\ 0 \end{array} \\
\hline 
N_6 & \begin{array}{c} (x_0 + a y_0 + ar+bs , y_0, z_0+b, 1)\\0 \end{array} \\
\hline
N_7 &  \begin{array}{c}(x_0 +ar+bs, y_0+ b,z_0+ a,1 ) \\0  \end{array} \\
\hline
N_8 & \begin{array}{c}(x_0 + ay_0 + bz_0 + ar+bs, y_0, z_0, 1) \\0 \end{array} \\
\hline
\end{array}$$
\end{proof} 

Notice Theorem \ref{4groups} provides an alternate proof that $N_4$ and $N_4'$ are not conjugate: $N_4'$ has a 2-dimensional subgroup with convex orbit, and $N_4$ does not.

The cusp Lie group ${C}([r:s:t])$ has a convex orbit if $rst(r+s+t)>0$. The convexity condition is shown in figure \ref{4tri} in an affine patch where $t=1$. 

\begin{figure}[h]
\begin{center}
\includegraphics[scale=.5]{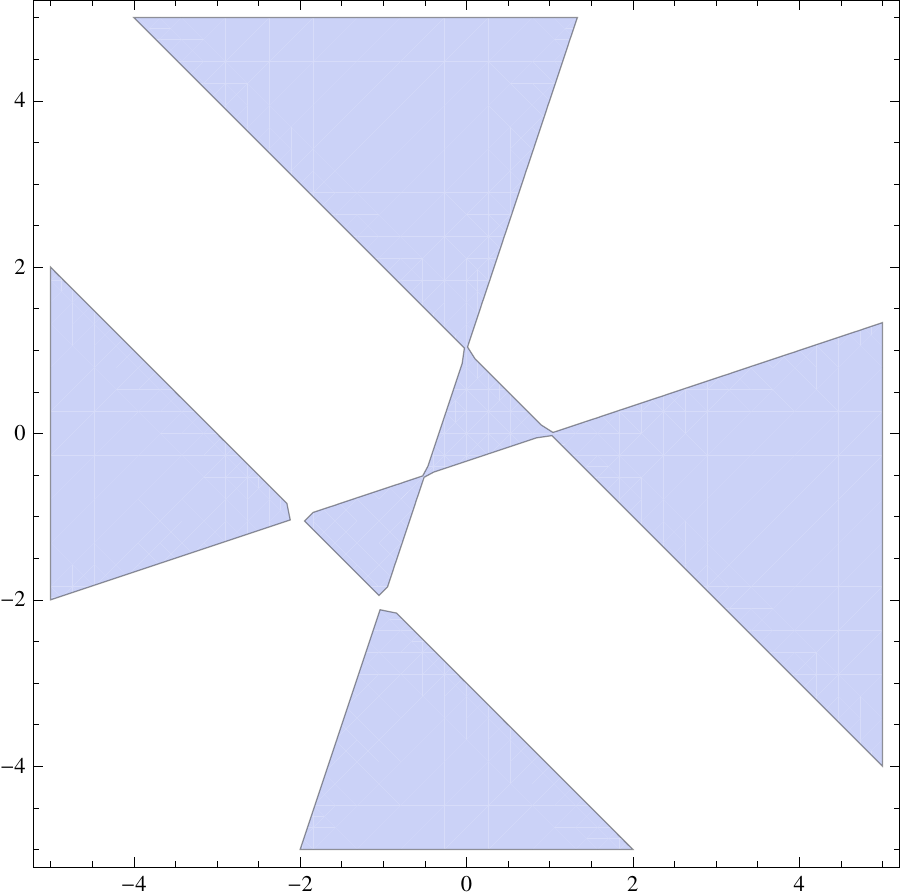}
\caption{  $rs(1+r+s)>0$ if and only if $(r,s)$ is in a shaded region. 
}
\label{4tri}
\end{center}
\end{figure}   

\section{Conjugacy Classification of Cusp Lie Groups }
 
This section determines which of the groups in Proposition \ref{4groups} are conjugate.  
The cusp Lie groups ${C}([r:s:t])$, ${E}(r,s)$, ${F}[r:s:-1]$, and ${N_4'}[r:s:-1]$ are not conjugate since they have different numbers of repeated weights. 
This section gives conditions for when two groups in the same family are conjugate, and parametrizes conjugacy classes in each family by subsets of projective space. This will conclude the proof of Theorem \ref{4cusps}. 

Since convexity of the cusp Lie group $F[r:s:-1]$ depends only on $r$, define $F(r) := F[r:0:-1]$. Since $N_4'[r:s:-1]$ does not depend on $r,s$, notice $N= N_4'[1:1:-1]$. 





\begin{prop}\label{conjsubset} 
\begin{enumerate}
\item Every cusp Lie group ${C}([r:s:t])$ is conjugate to $C([r':s':t'])$ for some $r' \geq s' \geq t' >0$.
\item Every cusp Lie group ${E}(r,s)$ is conjugate to $E(1, s')$ for some $1/2 > s' \geq 0$. 
\item  If $r \neq 0$, then $ {F}(r)$ is conjugate to ${F}(1)$. 
\end{enumerate}

\end{prop}

\begin{center}
\begin{SCfigure}[][h]\label{funddom} 
\includegraphics[scale=.2]{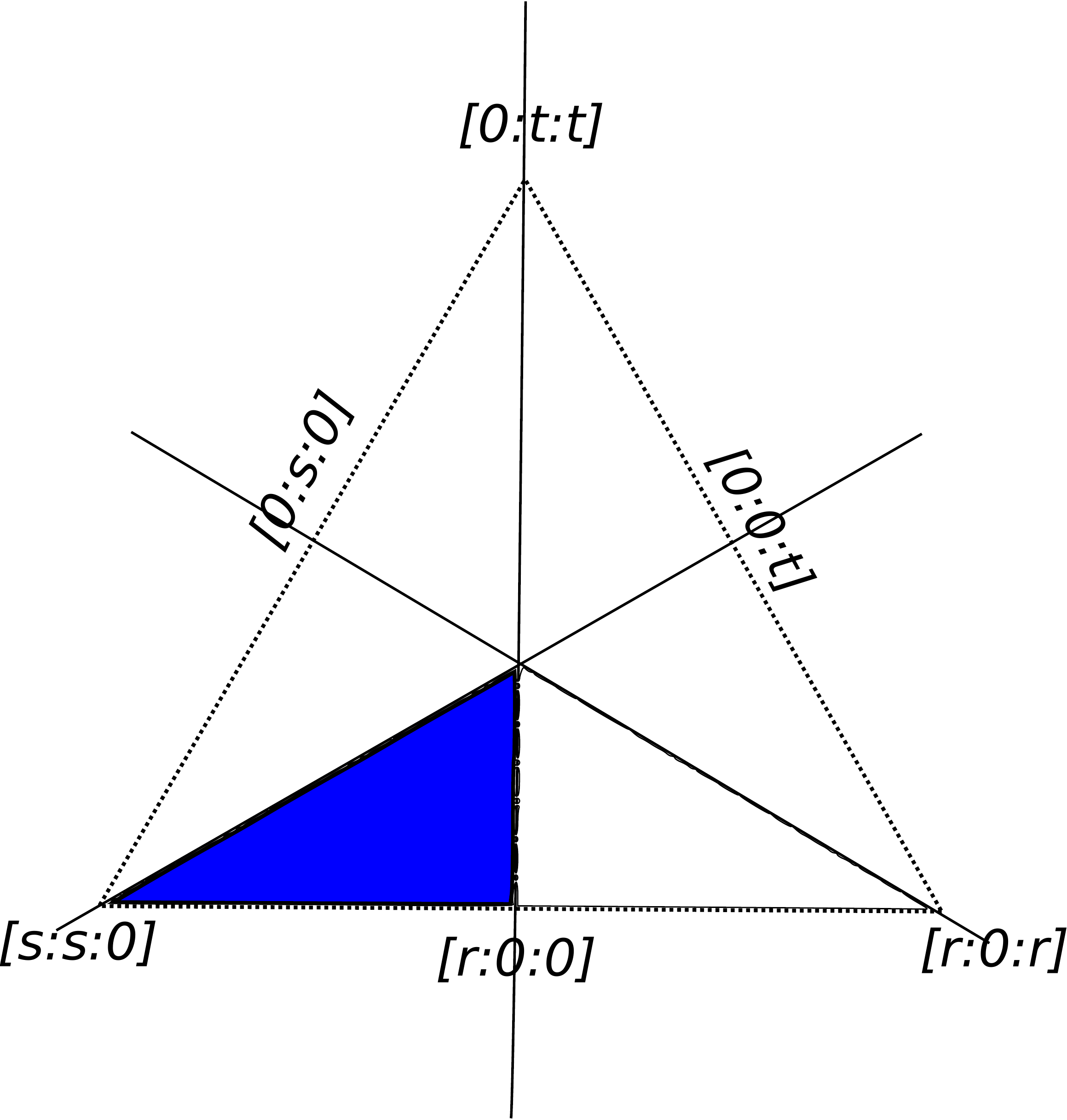}
\caption{A fundamental domain for $S_4 \curvearrowright \R P^2$,  where $r \geq s \geq t >0$}
\end{SCfigure}
\end{center}

\begin{proof}
 \begin{description}

\item[Conjugacy in] ${C}([r:s:t])$: 
Work in the setting of Lie algebras.  Let $\mathfrak{a} \cong \R^3$ be the Cartan subalgebra and $\mathfrak{a}^\ast$ be the dual.  So
$$ \mathfrak{a} =\Bigg \{
\left( \begin{array} {cccc}
x_1 & 0&0&0 \\
0&x_2 &0 &0 \\
0&0& x_3 &0 \\
0&0&0&x_4
\end{array} \right)  \Bigg |  x_i \in \R, \quad  \Sigma x_i =0 \Bigg \}. $$


Let $\phi \in \mathfrak{a}^\ast$ be a linear functional.  Then $\textrm{ker} \phi$ is a 2-dimensional subalgebra of $\mathfrak{a}$.   Notice $\textrm{ker} \phi$ is unchanged by scaling $\phi$. There is a bijection between points of $\R P^2 \cong \PR( \mathfrak{a} ^\ast)$ and 2-dimensional subalgebras of $\mathfrak{a}$.  We find the subset of $\R P^2$ which parametrizes conjugacy classes of 2-dimensional subalgebras.    Given $[r:s:t] \in \R P^2$, recall the Lie algebra
$${\mathfrak{c}}([r:s:t])= \Bigg \{
\left( \begin{array} {cccc}
x_1 & 0&0&0 \\
0&x_2 &0 &0 \\
0&0& x_3 &0 \\
0&0&0&-(x_1 + x_2 + x_3) 
\end{array} \right)
\Bigg|  rx_1 + s x_2 + tx_3 =0  \Bigg \}.
$$
By Proposition \ref{4groups}, ${C}([r:s:t])$ is convex when 
\begin{equation}\label{conv}
rst(r+s+t)>0.
\end{equation}
  We will describe a fundamental domain in $\R P^2$ which parametrizes convex subgroups ${C}([r:s:t])$. 
Conjugacy must permute the weight spaces, so ${C}([r:s:t])$ is conjugate to ${C}([r':s':t'])$ only if there is some $P \in GL_4(\R)$ which is a signed permutation of the standard basis of $\R ^4$, and ${C}([r:s:t]) = P {C}([r':s':t']) P^{-1}$. 

 Let $\{ e_1, e_2, e_3, e_4\}$ be the coordinate vectors in $\R ^4$, then $S_4$ acts on this set by  permutations.  Thus $S_4$ preserves the Cartan algebra, $\langle e_1 + e_2 + e_3 +e_4 \rangle ^ \perp = \mathfrak{a} \cong \R ^3$.  Define $f_i \in \R^3$ to be the orthogonal projection of $e_i$ onto $\mathfrak{a}$.  So $f_1 + f_2 + f_3 +f_4 =0$.  The $f_i$ are the vertices of a regular tetrahedron, $T$, centered at the origin.  The action of $S_4$ permutes $ \{ f_1, f_2, f_3, f_4 \}$.  So $S_4$ acts on $\mathfrak{a} \cong \R^3$ as the group of symmetries of $T$.    Therefore the subset of $\R P^2$ for which ${C}([r:s:t])$ is convex is divided into 24 fundamental domains under the action of $S_4$. 




Perform a change of coordinates: 
$$ \alpha_1 = r, \alpha_2 = s, \alpha_3 =t, \alpha_4 = -(r+s+t) $$
then the convexity condition becomes 
\begin{equation}\label{newconv} 
\alpha_1 \alpha_2 \alpha_3 \alpha_4 <0.
\end{equation}
   Notice 
   \begin{equation}\label{invtsum} 
   \alpha_1 + \alpha_2 + \alpha_3 + \alpha_4 =0,
   \end{equation}
 and \eqref{invtsum} is preserved under the action of $S_4$ on $\R ^4$.  In fact, this is the only linear equation preserved by the action of $S_4$. 
 
 Now $S_4$ acts transitively on $\{e_1, ..., e_4\}$, so $S_4$ acts transitively on the projection of $\{e_1,..., e_4\}$ in the plane \eqref{invtsum}.  This divides the double cover of $\PR(\mathfrak{a}^\ast)\cong \R  P^2$ into 14 regions as follows.  Let $S^2 \subset \R^3$ be the unit sphere, with the tiling of a cubeoctahedron shown in figure \ref{cuboct}.  There are 8 triangular regions where $rst(r+s+t)>0$ and ${C}([r:s:t])$ is convex; and 6 square regions where ${C}([r:s:t])$ is not convex.  The great circles correspond to $r=0, s=0, t=0,$ and $r+s+t=0$.

 \begin{figure}[h] 
  \begin{center}
 \includegraphics[scale=.15]{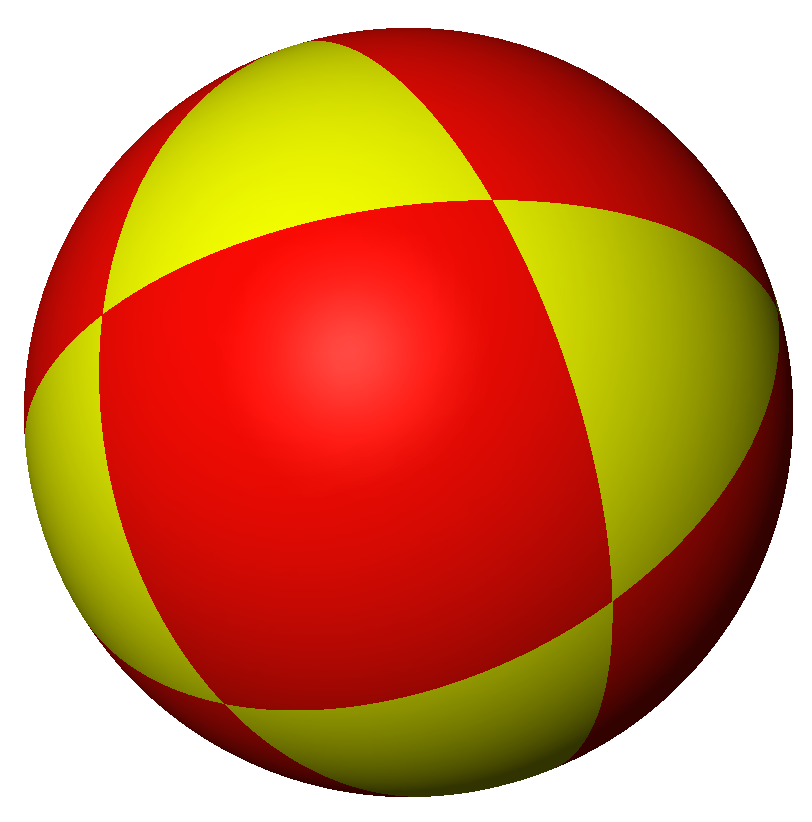}
 \caption{A cubeoctahedron.}  
   \label{cuboct}
  \end{center}
 \end{figure}

 The group of symmetries of a cubeoctahedron is the same as the group of symmetries of the cube: signed $4 \times 4$ permutation matrices.    Projecting down to $\R P^2$,  the group $S_4$ acts transitively on the triangles in figure \ref{4tri}, and a fundamental domain for the action is pictured in figure \ref{funddom}, where  $r \geq s \geq t >0$. 
Therefore every cusp Lie group $C([r:s:t])$ is conjugate to a cusp Lie group with $r \geq s \geq t >0$.

\item[Conjugacy in ${E}(r,s)$ :]
If $(r,s) \neq (0,0)$, then $E(r,s)$ has 3 eigenvectors, $\{ e_1, e_2, e_4 \}$, so any conjugacy must permute them.  
The dimension of the generalized eigenspace associated to $e_2$ is 2, which is larger than the dimension of the generalized eigenspaces associated to $e_1$ and $e_4$. So, any conjugacy must fix $e_2$, permute $\{e_1, e_4\}$, and sends $e_3$ to any vector in the the generalized eigenspace spanned by $\langle e_2, e_3 \rangle$.  Thus any conjugacy is 
by $Q$ or $QP$ where 
$$ Q= \left( \begin{array}{cccc}
 \alpha_1&0&0&0\\
 0& \alpha_2& \alpha_3 &0\\
 0&0& \alpha_4 &0\\
 0&0&0&\alpha_5
 \end{array} \right),
 \textrm{  and  } 
P=  \left( \begin{array}{cccc}
 0&0&0&1\\
 0& 1& 0 &0\\
 0&0& 1 &0\\
 1&0&0&0
 \end{array} \right), $$
 with $\alpha_1,..., \alpha_5 \in \R$. 
Since  $ Q E(r,s) Q^{-1}= {E}(\frac{\alpha_2}{\alpha_4}r,\frac{\alpha_2}{\alpha_4}s)$, and $[r:s]= [\frac{\alpha_2}{\alpha_4}r,\frac{\alpha_2}{\alpha_4}s] \in \R P^1$, conjugating by $Q$ does not change the group, and any conjugacy must be by $P$. 

Notice ${E}(0,0)$ is conjugate to ${C}([-2:-1:0])$, so assume $(r,s) \neq (0,0)$, and $[r:s] \in \R P^1$. 
Finally $E(r,s)$ is conjugate to 
$E(-r,s)$ by $P$. 
Since $(r-2s)(r+2s)>0$, every ${E}(r,s)$ is conjugate to a group where $r=1$ and $1/2>s\geq 0$.

\item[Conjugacy in ${F}(r)$:] Given $r,s \in \R$, the group  
$${F}(r,s) = \left( \begin{array}{cccc}
 e^a & e^a b & \frac{1}{2}e^a(b^2 + 2ar+2bs)&0 \\
 0 & e^a & e^a b &0  \\
 0 & 0 & e^a &0\\
 0&0&0& e^{-3a}  \end{array} \right), $$
is the image of $\mathfrak{f_1}[r:s:-1]$ under the exponential map, defined preceding Proposition \ref{4groups}.  First check ${F}(r,s)$ is conjugate to ${F}(r) $ by conjugating by the following matrix $S$. 

 $$
 S =  \left( \begin{array}{cccc}
1 &0&0&0\\
 0& 1 & s &0\\
 0&0&1 &0\\
 0&0&0&1
 \end{array} \right), 
 R =  \left( \begin{array}{cccc}
\frac{1}{\sqrt{r}} &1&1&0\\
 0& 1 & \sqrt{r}&0\\
 0&0&\sqrt{r} &0\\
 0&0&0&1
 \end{array} \right). 
 $$
If $r \neq 0$, then  ${F}(r)$ is conjugate to $ {F}(1)$ by $R$.   Recall from Proposition \ref{4groups} that $F(r)$ has a convex orbit if and only if $r >0$.  Thus there is only one cusp in this family.  
\end{description} \end{proof} 

This work is part of the author's PhD thesis.
\emph{Daryl Cooper} provided the inspiration for the project and Proposition \ref{fund_form}, and sparked many insightful discussions.  The author thanks \emph{Sam Ballas} and \emph{Jeff Danciger} for several helpful conversations. The referee provided excellent suggestions for improving the clarity of the paper. This work was partially supported by NSF grants DMSÐ0706887, 1207068 and 1045292, and by U.S. National Science Foundation grants DMS 1107452, 1107263,
1107367 RNMS: \emph{GEometric structures And Representation varieties (the GEAR Network).}


\end{document}